\documentclass[11pt]{article}
\date{}

\usepackage[title]{appendix}
\usepackage{xcolor}
\usepackage{geometry}                
\usepackage{graphicx}
\usepackage{subcaption}
\usepackage{pdflscape}
\usepackage{amssymb}
\usepackage[normalem]{ulem}
\usepackage{hyperref}
\usepackage{enumitem}
\usepackage{mathrsfs}
\usepackage{epstopdf}
\usepackage{rotating}
\usepackage{longtable} 
\usepackage{adjustbox}
\usepackage{float}

\usepackage{color}
\usepackage{bbm, dsfont}
\usepackage{pst-node}
\usepackage{tikz-cd}
\usepackage{amsfonts} 
\usepackage{geometry}
\usepackage{amsthm}
\usepackage{amsmath}
\usepackage{titlesec}
\usepackage{fixltx2e}
\usepackage{amssymb}
\usepackage{enumitem}
\usepackage{float}
\usepackage [english]{babel}
\usepackage [autostyle, english = american]{csquotes}
\usepackage{algorithm}
\usepackage[noend]{algpseudocode} 
\makeatletter
\def\BState{\State\hskip-\ALG@thistlm}
\makeatother
\usepackage{hyperref}
\usepackage[normalem]{ulem}

\newlist{casess}{enumerate}{1}
\setlist[casess]{label=     \textbf{Case} \arabic*:}
\usepackage{mathtools}

\makeatletter
\newcommand*{\rom}[1]{\expandafter\@slowromancap\romannumeral #1@}
\makeatother

\usepackage{etoolbox}

\makeatletter
\patchcmd{\ttlh@hang}{\parindent\z@}{\parindent\z@\leavevmode}{}{}
\patchcmd{\ttlh@hang}{\noindent}{}{}{}
\makeatother

\usepackage{listings}
\usepackage{color} 
\definecolor{mygreen}{RGB}{28,172,0} 
\definecolor{mylilas}{RGB}{170,55,241}

\newlist{Assumptions}{enumerate}{1}
\setlist[Assumptions]{label=     \textbf{Assumption} \arabic*:}

\makeatletter

\newsavebox{\@brx}
\newcommand{\llangle}[1][]{\savebox{\@brx}{\(\m@th{#1\langle}\)}%
  \mathopen{\copy\@brx\kern-0.5\wd\@brx\usebox{\@brx}}}
\newcommand{\rrangle}[1][]{\savebox{\@brx}{\(\m@th{#1\rangle}\)}%
  \mathclose{\copy\@brx\kern-0.5\wd\@brx\usebox{\@brx}}}
\makeatother

\usepackage{lipsum} 
\usepackage{titlesec}
\titleformat{\subsection}[runin]
       {\normalfont\bfseries}
       {\thesubsection}
       {0.5em}
       {}
       [.]

\newcommand{\A}{\mathfrak{A}}

\newcommand{\B}{\mathfrak{B}} 

\newcommand{\CC}{\mathbb{C}}

\usepackage{tikz-cd}
\usepackage{graphicx}

\def\N{\mathbb{N}}

\def\R{\mathbb{R}}

\def\X{\mathcal{X}}

\def\e{{\sf e}}

\def\BofH{\mathbb B(\mathcal H)}

\def\d{{\rm d}}
\def\bu{\bullet}

\def\({\left(}
\def\[{\left[}
\def\){\right)}
\def\]{\right]}

\def\G{{\sf G}}

\def\<{\langle}
\def\>{\rangle}
\providecommand{\norm}[1]{\lVert#1\rVert}

 \newtheorem{thm}{Theorem}[section]
 \newtheorem{cor}[thm]{Corollary}
 
 \newtheorem{lem}[thm]{Lemma}
 \newtheorem{prop}[thm]{Proposition}
 \theoremstyle{definition}
 \newtheorem{defn}[thm]{Definition}
 \theoremstyle{remark}
 \newtheorem{rem}[thm]{Remark}
 \newtheorem{ex}[thm]{Example}
 \numberwithin{equation}{section}

\numberwithin{equation}{section}

\usepackage[backend=biber, maxnames=10]{biblatex}
\begin{document}


\title{On the continuity of intertwining operators over generalized convolution algebras}

\author{Felipe I. Flores}

\author{Felipe I. Flores
}

\maketitle


\begin{abstract}
Let $\G$ be a locally compact group, $\mathscr C\overset{q}{\to}\G$ a Fell bundle and $\B=L^1(\G\,\vert\,\mathscr C)$ the algebra of integrable cross-sections associated to the bundle. We give conditions that guarantee the automatic continuity of an intertwining operator $\theta:\X_1\to\X_2$, where $\X_1$ is a Banach $\B$-bimodule and $\X_2$ is a weak Banach $\B$-bimodule, in terms of the continuity ideal of $\theta$. We provide examples of algebras where this conditions are met, both in the case of derivations and algebra morphisms. In particular, we show that, if $\G$ is infinite, finitely-generated, has polynomial growth and $\alpha$ is a free (partial) action of $\G$ on the compact space $X$, then every homomorphism of $\ell^1_\alpha(\G,C(X))$ into a Banach algebra is automatically continuous.
\end{abstract}

\section{Introduction}\label{introduction}

Intertwining operators (or more precisely called, generalized intertwining operators) seem to have been introduced first in \cite{BD78} as a common generalization of algebra homomorphisms, derivations and bimodule homomorphisms. Thus they provide a unified framework for the study of automatic continuity of the three mentioned types of maps. The continuity of intertwining operators has been previously studied in \cite{La81,Ru96,Ru98}, while their connections with cohomology and extensions of Banach algebras have been studied in \cite{BDL99}. On the other hand, the automatic continuity of algebra homomorphisms and/or derivations has been studied in \cite{Ru94,Ru96c,Run94,Si74}, among others. Our reference for the general theory of automatic continuity is \cite{Da00}.

If $\B$ is a $C^*$-algebra or the group algebra $L^1(\G)$ of a compactly generated group with polynomial growth, then the continuity of an intertwining operator $\theta$ from $\B$ to a weak Banach $\B$-bimodule can be characterized as the closedness of the continuity ideal $\mathscr I(\theta)$ \cite{Ru96}. However, this ideal is always closed when the codomain is a Banach $\B$-bimodule, thus implying that derivations and $\B$-module maps are automatically continuous. Having this information at hand, it seems natural to ask about the extent to which the same results hold for (twisted) convolution algebras $L^1_{\alpha,\omega}(\G,\A)$, as they generalize both classes of algebras. These algebras maintain enough of the flavor of a group algebra to be approached in a similar fashion but they are also able to exhibit new phenomena. 

In fact, by taking advantage of the author's latest results \cite{Fl24} we are able to handle the general setting given by $L^1$-algebras of Fell bundles, thus providing a very general framework for our results that encompasses both classical crossed-product-type algebras and their generalizations. Our approach makes use of weights on the group, smooth functional calculus and both the ideal and representation theories of $L^1(\G\,\vert\,\mathscr C)$ and so it is highly related to the more classical study of harmonic analysis on groups. Many of our results on $L^1$-algebras could be considered of independent interest for the researchers in this area.

In any case, our main result is the following.
\begin{thm}
    Let $\G$ be a locally compact group that admits a polynomial weight $\nu$, such that $\nu^{-1}\in L^p(\G)$. Suppose that the algebra $\B=L^1(\G\,\vert\,\mathscr C)$ is symmetric and that every closed two-sided ideal $I$ of $\B$ has a bounded left approximate identity. Let $\X$ a weak Banach $\B$-bimodule and $\theta:\B\to\X$ a $\B$-intertwining operator. Then $\theta$ is continuous if and only if $\mathscr I(\theta)$ is closed.
\end{thm}

We also obtained the following results, which are versions of the previously stated theorem, but seem interesting by their own right. We use them to provide examples of algebras where all derivations are automatically continuous (see Corollary \ref{segundo}) and algebras where all homomorphisms (into other Banach algebras) are continuous (see Corollary \ref{tercero}). 

\begin{cor}
    Let $\G$ be a locally compact group that admits a polynomial weight $\nu$, such that $\nu^{-1}\in L^p(\G)$. Suppose that the algebra $\B=L^1(\G\,\vert\,\mathscr C)$ is symmetric and that every closed two-sided ideal $I$ of $\B$ has a bounded left approximate identity. Let $\X$ a Banach $\B$-bimodule and $\theta:\B\to\X$ a $\B$-intertwining operator. Then $\theta$ is continuous. In particular, all derivations of $\B$ into Banach $\B$-bimodules are continuous.
\end{cor}

\begin{thm}
    Let $\G$ be a discrete group that admits a polynomial weight $\nu$, such that $\nu^{-1}\in \ell^p(\G)$. Suppose that the algebra $\B=\ell^1(\G\,\vert\,\mathscr C)$ is unital, symmetric and $C^*(\G\,\vert\,\mathscr C)$ has no proper closed two-sided ideals with finite codimension. Let $\X_1$ a Banach $\B$-bimodule, $\X_2$ a weak Banach $\B$-bimodule and $\theta:\X_1\to\X_2$ a $\B$-intertwining operator. Then $\theta$ is continuous. In particular, all algebra homomorphisms with domain $\B$ are automatically continuous.
\end{thm}

The organization of the article is as follows. Section \ref{prem} contains preliminaries. It is basically used to compile basic definitions from Banach algebra theory and to fix some notation. Section \ref{AC} contains general results from the theory of automatic continuity. Most of them are taken from pre-existing literature, as explained along the section itself. The results in \cite{Ru96} originally stated for group algebras are put into general form, so they can be applied (in particular) to the algebras of our interest. On Section \ref{mainsec} we study the $L^1$-algebras of Fell bundlea in full generality, with the purpose of applying them the results obtained in the previous section. This study involves weighted subalgebras, $^*$-regularity, the use of the smooth functional calculus developed in \cite{Fl24} and spectral synthesis. In this section we prove the main results as stated above and provide the examples where they apply. We also prove some interesting results like Corollary \ref{noncont}, providing more examples where the Albrecht-Dales conjecture holds. Finally, Section \ref{COMPL} is an appendix where we verify that finite-codimensional closed two-sided ideals in twisted group algebras have left bounded approximate identities, as long as the group is amenable. This is useful in providing more examples for our main results, but it could easily regarded as of independent interest.

\section{Preliminaries}\label{prem}

As mentioned before, the purpose of this section is to fix notation and terminology. If $\B$ is a Banach algebra, $\B(b_1,\ldots, b_n)$ denotes the closed subalgebra of $\B$ generated by the elements $b_1,\ldots, b_n\in\B$. The set of invertible elements in $\B$ is denoted by ${\rm Inv}(\B)$. If $\B$ has an involution, $\B_{\rm sa}$ denotes the set of self-adjoint elements in $\B$, that is, of all $b\in\B$ such that $b^*=b$. A Banach $^*$-algebra admiting a $C^*$-norm is called an $A^*$-algebra. 

\begin{defn}
    Let $\B$ be a Banach $^*$-algebra. If $\B$ is unital, we set $\widetilde{\B}=\B$. Otherwise, $\widetilde{\B}=\B\oplus \CC$ is the smallest unitization of $\B$, endowed with the norm $\norm{b+r1}_{\widetilde{\B}}=\norm{b}_{{\B}}+|r|$.
\end{defn}

\begin{defn} \label{symmetric}
A Banach $^*$-algebra $\mathfrak B$ is called {\it symmetric} if the spectrum of $b^*b$ is positive for every $b\in\mathfrak B$ (this happens if and only if the spectrum of any self-adjoint element is real).
\end{defn}

If $\B$ is a commutative Banach algebra with spectrum $\Delta_\B$, then $\hat b\in C_0(\Delta_\B)$ denotes the Gelfand transform of $b\in \B$. If the Gelfand transform is injective, $\B$ is called a Banach function algebra. 

\begin{defn}
    Let $\B$ be a Banach function algebra with spectrum $\Delta_\B$. $\B$ is called {\it regular} if for every closed set $X\subset \Delta_\B$ and every point $\omega\in \Delta_\B\setminus X$, there exists an element $b\in \B$ such that $\hat b(\varphi)=0$ for all $\varphi\in X$ and $\hat b(\omega)\not=0$.
\end{defn}

\begin{rem}\label{regsym}
    It is a result of Barnes \cite[Lemma 4.1]{Ba81} that regular Banach function algebras with involution are symmetric.
\end{rem}

\begin{defn}
    An $A^*$-algebra $\B$ is called \emph{locally regular} if there is a subset $R\subset \B_{\rm sa}$, dense in $\B_{\rm sa}$ and such that $\B(b)$ is regular, for all $b\in R$.
\end{defn}

If $\B$ is a Banach algebra, then the spaces ${\rm Prim}\B$ and ${\rm Prim}_*\B$ denote, respectively, the space of primitive ideals of $\B$ and the space of kernels of topologically irreducible
$^*$-representations of $\B$, both equipped with the Jacobson topology. We recall that for a subset $S\subset \B$, its hull (with respect to ${\rm Prim}_*\B$) corresponds to $$h(S)=\{I\in {\rm Prim}_*\B\mid S\subset I\},$$ while the kernel of a subset $C\subset {\rm Prim}_*\B$ is $$k(C)=\bigcap_{I\in C} I.$$ Similar formulas hold for the hull and kernel with respect to ${\rm Prim}\B$. 

We now consider the following property, which is intended as an abstract generalization of Wiener's tauberian theorem. It could be regarded in some sense as the existence of enough topologically irreducible
$^*$-representations of $\B$.

\begin{defn}
    Let $\B$ be a Banach $^*$-algebra. We say that $\B$ has the Wiener property $(W)$ if for every proper closed two-sided ideal $I\subset\B$, there exists a topologically irreducible $^*$-representation $\Pi:\B\to\BofH$, such that $I \subset {\rm ker}\,\Pi$. 
\end{defn}

We finalize this section with a property of the topological space ${\rm Prim}_*\B$ that will become handy soon.

\begin{defn}
    A closed subset $C\subset {\rm Prim}_*\B$ will be called a \emph{set of synthesis} if $k(C)$ is the unique closed two-sided ideal $I\subset \B$ such that $h(I)=C$.
\end{defn}

\section{Automatic continuity}\label{AC}

Let $\B$ be Banach algebra. A Banach space $\mathcal X$ which is also a $\B$-bimodule is called a \emph{Banach $\B$-bimodule} if the maps $$\B\times \mathcal X\ni(b,\xi)\mapsto b\xi \in\mathcal X \quad\text{ and }\quad \X\times\B\ni(\xi,b)\mapsto \xi b \in\mathcal X$$ are jointly continuous. If we only have the continuity of the maps $$\mathcal X\ni\xi\mapsto b\xi \in\mathcal X\quad\text{ and }\quad\mathcal X\ni\xi\mapsto \xi b \in\mathcal X$$ for each $b\in\B$, then $\mathcal X$ is called a \emph{weak Banach $\B$-bimodule}.

\begin{defn}
    Let $\B$ be a Banach algebra and $\X_1,\X_2$ be weak Banach $\B$-bimodule. A linear map $\theta:\X_1\to\X_2$ is called a \emph{$\B$-intertwining operator} if for each $b\in\B$, the maps $$\X_1\ni \xi\mapsto \theta(b\xi)-b\theta(\xi)\in \X_2\quad\text{ and }\quad \X_1\ni \xi\mapsto \theta(\xi b)-\theta(\xi)b\in \X_2$$ are continuous.
\end{defn}

We mentioned in the introduction that intertwining operators generalize algebra homomorphisms, derivations and bimodule homomorphisms. We will make this precise in the following example.
\begin{ex}\label{ex-inter}\begin{enumerate}
        \item[(i)] Every $\B$-bimodule homomorphism between weak Banach $\B$-bimodules is a $\B$-intertwining operator.
        \item[(ii)] If $\theta:\B\to\A$ is an algebra homomorphism, then $\A$ can be made into a weak Banach $\B$-bimodules with the actions $$ab=a\theta(b)\quad\text{ and }\quad ba=\theta(b)a, \quad\text{for } a\in \A, b\in \B.$$ With respect to this actions, $\theta$ is a $\B$-intertwining operator.
        \item[(iii)] Let $\X$ a weak Banach $\B$-bimodule. A derivation is a linear map $D:\B\to \X$ satisfying $$D(ab)=D(a)b+aD(b).$$ Every derivation is a $\B$-intertwining operator. 
    \end{enumerate}
\end{ex}

If $\X_1,\X_2$ are Banach spaces and $\theta:\X_1\to\X_2$ is a linear operator, then $$\mathscr S(\theta)=\{\eta\in\X_2\mid \exists \{\xi_n\}_{n\in \N}\subset \X_1 \text{ such that } \xi_n\to 0 \text{ and } \theta(\xi_n)\to \eta\}$$ is the \emph{separating space} of $\theta$. $\mathscr S(\theta)$ is closed and reduces to $\{0\}$ if and only if $\theta$ is continuous. If $\X_1,\X_2$ are weak Banach $\B$-bimodules and $\theta$ is $\B$-intertwining, $\mathscr S(\theta)$ is a sub-bimodule of $\X_2$.

\begin{defn}
    Let $\B$ be a Banach algebra, and $\theta:\X_1\to\X_2$ a $\B$-intertwining operator between weak Banach $\B$-bimodules. Then $$\mathscr I(\theta)=\{b\in\B\mid b\mathscr S(\theta)=\mathscr S(\theta)b=\{0\}\}$$ is the \emph{continuity ideal} of $\theta$.
\end{defn}

In fact, $\mathscr I(\theta)$ coincides with the set of $b\in\B$ such that the maps $$\X_1\ni \xi\mapsto \theta(b\xi)\in \X_2 \quad\text{and}\quad \X_1\ni \xi\mapsto \theta(\xi b)\in \X_2$$ are continuous. Note that $\mathscr I(\theta)$ is not necessarily closed, unless $\X_2$ is a Banach $\B$-bimodule.

In what follows, our strategy will be to provide assumptions that ensure that $\overline{\mathscr I(\theta)}^{\norm{\cdot}_{\B}}$ has finite codimension in $\B$. Our strategy is local and therefore we will use that this holds for regular Banach function algebras, as the following theorem (taken from \cite{BaCu60}) shows.

\begin{thm}\label{abnfunalg}
    Let $\A$ be a regular Banach function algebra, $\X_1$ a Banach $\A$-bimodule, $\X_2$ a weak Banach $\A$-bimodule and $\theta:\X_1\to\X_2$ an $\A$-intertwining operator. Then $h\big(\mathscr I(\theta)\big)$ is a finite subset of $\Delta_\A$.
\end{thm}

\begin{thm}\label{thm1}
    Let $\B$ be a Banach $^*$-algebra, $\X_1$ a Banach $\B$-bimodule, $\X_2$ a weak Banach $\B$-bimodule and $\theta:\X_1\to\X_2$ a $\B$-intertwining operator. Suppose that every self-adjoint element $b\in \B_{\rm sa}$ generates a regular Banach function algebra. Then $h\big(\mathscr I(\theta)\big)$ is empty or consists of a finite number of ideals, each of which has finite codimension in $\B$.
\end{thm}
\begin{proof}
    Let $\mathscr I(\theta)_*=\{a\in \mathscr I(\theta)\mid a^*\in\mathscr I(\theta)\}$. Then $\mathscr I(\theta)_*$ and its closure $\overline{\mathscr I(\theta)_*}$, are self-adjoint ideals of $\B$. Now, by assumption, $\B(b)$ is a regular Banach function algebra, for every $b=b^*\in\B$. Viewing $\theta$ as a $\B(b)$-intertwining operator, we can apply Theorem \ref{abnfunalg} to see that $S:=h\big(\B(b)\cap \mathscr I(\theta)\big)$ is a finite subset of $\Delta_{\B(b)}$. If we let $$J(S)=\{f\in\B(b)\mid {\rm Supp}(\hat f) \text{ is compact and does not intersect }S\},$$ then, because of the regularity of $\B(b)$, $h\big(J(S)\big)=S$ and $J(S)$ is contained in each ideal of $\B(b)$ whose hull is $S$. In particular, $J(S)\subset \B(b)\cap \mathscr I(\theta)$. By Remark \ref{regsym}, $\B(b)$ is also symmetric and hence $$f\in J(S) \Leftrightarrow f^*\in J(S),$$ for all $f\in \B(b)$ \cite[Theorem 11.4.1]{Pa94}. Consequently $J(S)\subset \B(b)\cap \mathscr I(\theta)_*$ and therefore $h\big( \B(b)\cap \mathscr I(\theta)_*\big)=S$. It then follows that $b+\overline{\mathscr I(\theta)_*}$ has finite spectrum in $\B/\overline{\mathscr I(\theta)_*}$. If we let $I$ be the intersection of all ideals in ${\rm Prim}_*\B$ containing $\overline{\mathscr I(\theta)_*}$, we first note that $\B/I$ is an $A^*$-algebra \cite[Theorem 9.7.10]{Pa94} and that for every $b\in\B_{\rm sa}$, $b+I\in \B/I$ has finite spectrum, as it is the homomorphic image of $b+\overline{\mathscr I(\theta)_*}$. By \cite[Corollary 5.4.3]{Au91}, the algebra $\B/I$ is finite-dimensional and, hence, there are only finitely many ideals in ${\rm Prim}_*\B$ containing $\overline{\mathscr I(\theta)_*}$ and all of them are of finite codimension.
\end{proof}

The following theorem -the main of this section- is a generalization of \cite[Theorem 2.3, Theorem 3.2]{Ru96}. Although the context here is somewhat different, the proof is exactly the same, and it will be repeated here for convenience. As in Example \ref{ex-inter}, any Banach algebra $\B$ will be considered a bimodule over itself with left/right multiplication.

\begin{thm}\label{theprop}
    Let $\B$ be an $A^*$-algebra, $\X_1$ a Banach $\B$-bimodule, $\X_2$ a weak Banach $\B$-bimodule and $\theta:\X_1\to\X_2$ a $\B$-intertwining operator. Further suppose that 
    \begin{enumerate}
        \item[(i)] There exists a dense Banach $^*$-subalgebra $\A\subset \B$, such that $\A(b)$ is a regular Banach function algebra, for all $b\in\A_{\rm sa}$.
        \item[(ii)] $\B$ has the Wiener property \emph{(W)}.
        \item[(iii)] Every finite subset $F\subset{\rm Prim}_*\B$ such that all $P\in F$ has finite codimension is a set of synthesis for $\B$.
    \end{enumerate} Then $\overline{\mathscr I(\theta)}^{\norm{\cdot}_{\B}}$ has finite codimension in $\B$. Furthermore, if $\X_1=\B$ and $\B$ also satisfies \begin{enumerate}
        \item[(iv)] Every closed two-sided ideal $I\subset \B$ of finite codimension has a bounded left approximate identity.
    \end{enumerate} Then $\theta$ is continuous if and only if $\mathscr I(\theta)$ is closed. \end{thm}

\begin{proof}
    As the inclusion $\A\subset \B$ is continuous, we can view $\X_1$ as a Banach $\A$-bimodule, $\X_2$ a weak Banach $\A$-bimodule and $\theta$ as an $\A$-intertwining operator, whose continuity ideal is $\A\cap \mathscr I(\theta)$. Now, if $\overline{\mathscr I(\theta)}^{\norm{\cdot}_{\B}}$ is a proper ideal, then there exists $P\in {\rm Prim}_*\B$ containing $\mathscr I(\theta)$. Since $\A$ is dense in $\B$, the ideal $\A\cap P$ belongs to ${\rm Prim}_*\A$ and contains $\A\cap \mathscr I(\theta)$. Now let $I$ be the intersection of all $P\in {\rm Prim}_*\B$ containing $\mathscr I(\theta)$ and note that $$\A\cap I=\bigcap_{P\in h(\mathscr I(\theta))} \A\cap P \supset\bigcap_{P'\in h(\A\cap\mathscr I(\theta))}P'=k(h\big(\A\cap\mathscr I(\theta)\big)).$$ Because of Theorem \ref{thm1}, $\A\cap I$ has finite codimension in $\A$. That means that the image of $\A$ under the quotient map $\B\to\B/I$ is finite-dimensional and, by density, $\B/I$ has to be finite-dimensional too. Consequently, $h\big(\mathscr I(\theta)\big)$ is a finite subset of ${\rm Prim}_*\B$, composed of ideals with finite codimension and therefore a set of synthesis. Hence $I=\overline{\mathscr I(\theta)}^{\norm{\cdot}_{\B}}$ and the first conclusion follows.

    We now complete the second part of the proof. If $\theta$ is continuous, $\mathscr I(\theta)$ is clearly closed. On the other hand, if $\mathscr I(\theta)$ is closed then  it has finite codimension and so, by assumption, it has a bounded left approximate identity. Now and because of the Cohen-Hewitt factorization theorem \cite[Corollary 11.12]{BD73}, for every sequence $\{b_n\}\subset\mathscr I(\theta)$ converging to zero, there exists $c, d_n\in\mathscr I(\theta)$ that factorize $b_n$: $$b_n=cd_n\quad\text{and}\quad \lim_{n} d_n=0.$$ Since the map $\B\ni d\mapsto \theta(cd)$ is continuous by the definition of $\mathscr I(\theta)$, we have $$\lim_n\theta(b_n)=\lim_n\theta(cd_n)=0$$ and thus the restriction of $\theta$ to $\mathscr I(\theta)$ is continuous. Since $\mathscr I(\theta)$ has finite codimension, $\theta$ is in fact continuous on all of $\B$.
\end{proof}

\begin{rem}\label{conditions}
Condition \emph{(iv)} in Theorem \ref{theprop} will be the most restrictive in what follows, so it seems convenient now to mention how to imply it. This condition is satisfied by $L^1(\G)$ when $\G$ is amenable \cite[Theorem 2]{LvJ73} and more generally, for twisted group algebras of amenable groups (Theorem \ref{twistedgp}). In the more abstract setting, it is also satisfied by $C^*$-algebras and amenable Banach algebras \cite[Proposition VII.2.31]{He89}.
\end{rem}

Now we will consider a case new to the setting of convolution algebras, namely the absence of finite codimensional two-sided closed ideals. When one considers the group algebra $L^1(\G)$, finite codimensional two-sided closed ideals always exist. An easy example of that is the kernel of the augmentation map $$L^1(\G)\ni\Phi\mapsto\int_\G\Phi(x)\d x\in\CC,$$ which is called the augmentation ideal. Furthermore, if $\G$ is abelian, plenitude of such ideals exist, as the study of $C^*(\G)\cong C_0(\hat \G)$ easily indicates. On the other hand, it is perfectly possible for $L^1(\G\,\vert\,\mathscr C)$ to be simple. In fact, it can be showed -with the same arguments that we will use in the proof of Proposition \ref{simple}- that if $L^1(\G\,\vert\,\mathscr C)$ is simple if it is unital, symmetric and $C^*(\G\,\vert\,\mathscr C)$ is simple.

\begin{prop}\label{simple}
    Let $\B$ be a symmetric, unital $A^*$-algebra. Suppose $C^*(\B)$ has no proper closed two-sided ideals with finite codimension. Then the same holds for $\B$. 
\end{prop}

\begin{proof}
    Let $I$ a finite-codimensional closed two-sided ideal of $\B$. Then $I$ must be dense in $C^*(\B)$, as its closure is a finite-codimensional two-sided ideal of $C^*(\B)$. In particular, there exists a sequence $b_n\in I$ such that $b_n\to 1$ and therefore $b_n$ is invertible in $C^*(\B)$, for a large enough $n$. But in that case $b_n$ is invertible in $\B$ too and hence $I=\B$. 
\end{proof}

The ideal theory of $C^*$-algebras is both more studied and more forgiving, so taking the condition in Proposition \ref{simple} as an assumption seems reasonable. In fact, the same condition has been used previously by Sinclair \cite[Corollary 3.4]{Si74} and by Albrecht and Dales \cite[Theorem 2.6]{AlDa83}. We will use it in the following proposition to guarantee the automatic continuity of a big class of intertwining operators over $\B$.

\begin{cor}\label{theprop-simple}
    Let $\B$ be a symmetric, unital $A^*$-algebra, $\X_1$ a Banach $\B$-bimodule, $\X_2$ a weak Banach $\B$-bimodule and $\theta:\X_1\to\X_2$ a $\B$-intertwining operator. Further suppose that 
    \begin{enumerate}
        \item[(i)] There exists a dense Banach $^*$-subalgebra $\A\subset \B$, such that $\A(b)$ is a regular Banach function algebra, for all $b\in\A_{\rm sa}$.
        \item[(ii)] $\B$ has the Wiener property \emph{(W)}.
        \item[(iii)] $C^*(\B)$ has no proper closed two-sided ideals with finite codimension.
    \end{enumerate} Then $\theta$ is continuous.
\end{cor}
\begin{proof}
    Repeating the first part of the proof of Theorem \ref{theprop} shows that, if $\overline{\mathscr I(\theta)}^{\norm{\cdot}_{\B}}$ was a proper ideal, then $h\big(\mathscr I(\theta)\big)$ would be a finite subset of ${\rm Prim}_*\B$, composed of ideals with finite codimension. However, this is not possible due to Proposition \ref{simple}. That makes $\mathscr I(\theta)$ a dense two-sided ideal in $\B$. As $\B$ is unital, basic spectral theory tells us that $\mathscr I(\theta)$ must contain an invertible element and therefore $\mathscr I(\theta)=\B$. So $\theta$ must be continuous.
\end{proof}

\begin{rem}
    While is true that in order to apply Corollary \ref{theprop-simple} we added the condition of symmetry, in practice this is not a restriction as we will need to assume symmetry to guarantee the Wiener property. More on this in the next section.
\end{rem}

The following theorem gives conditions that allows to check the continuity of homomorphisms of $C^*(\B)$ at the level of $\B$. It is the main result of \cite{Run94}.
\begin{thm}\label{cstar}
    Let $\B$ be a symmetric, locally regular $A^*$-algebra with the following properties: \begin{enumerate}
        \item[(i)] Every closed two-sided ideal $I\subset\B$ of finite codimension has a bounded left approximate identity.
        \item[(ii)] Every closed, finite subset $F\subset {\rm Prim}_*\B$ that only consists of ideals of finite codimension is a set of synthesis for $\B$.
    \end{enumerate} Let $\theta:C^*(\B)\to \A$ be an homomorphism, where $\A$ is another Banach algebra. Then the following are equivalent: \begin{enumerate}
        \item[(a)] $\theta$ is continuous.
        \item[(b)] $\theta|_\B$ is continuous.
        \item[(c)] $\mathscr I(\theta)\cap \B$ is closed. 
    \end{enumerate}
\end{thm}

\section{An application to $L^1$-algebras associated with Fell bundles}\label{mainsec}

From now on $\G$ will be a (Hausdorff) unimodular, locally compact group with unit $\e$ and Haar measure $d\mu(x)\equiv dx$. If $\G$ is compact, we assume that $\mu$ is normalized so that $\mu(\G)=1$. We recall that $\G$ has \emph{polynomial growth of order $d$} if $$
\mu(K^n)=O(n^d),\quad\textup{ as }n\to\infty,
$$ for all relatively compact subsets $K\subset\G$. We will also fix a Fell bundle $\mathscr C\!=\bigsqcup_{x\in\G}\mathfrak C_x$ over $\G$. The algebra of integrable cross-sections $L^1(\G\,\vert\,\mathscr C)$ is a Banach $^*$-algebra and a completion of the space $C_{\rm c}(\G\,\vert\,\mathscr C)$ of continuous cross-sections with compact support. Its (universal) $C^*$-algebra its denoted by ${\rm C^*}(\G\,\vert\,\mathscr C)$. For the general theory of Fell bundles we cite \cite[Chapter VIII]{FD88}, to which we refer for details. We will only recall the product on $L^1(\G\,\vert\,\mathscr C)$, given by
\begin{equation}\label{broduct}
\big(\Phi*\Psi\big)(x)=\int_\G \Phi(y)\bu \Psi(y^{-1}x)\,\d y
\end{equation}
and its involution
\begin{equation}\label{inwol}
\Phi^*(x)=\Phi(x^{-1})^\bu\,,
\end{equation}
in terms of the operations $\big(\bu,^\bu\big)$ on the Fell bundle. We will make use of the $L^p$-spaces $L^p(\G\,\vert\,\mathscr C)$, endowed with the norms \begin{equation}
    \norm{\Phi}_{L^p(\G\,\vert\,\mathscr C)}=\left\{\begin{array}{ll}
\,\big(\int_\G \norm{\Phi(x)}_{\mathfrak C_x}^p \d x\big)^{1/p}    & \textup{if\ } p\in[1,\infty), \\
\,{\rm essup}_{x\in \G}\norm{\Phi(x)}_{\mathfrak C_x}     & \textup{if\ } p=\infty. \\
\end{array}\right.
\end{equation} 

The next example introduces one of the main classes of algebras we wish to study. 

\begin{ex}\label{mainex}
    Let $\A$ be a $C^*$-algebra. A (continuous) twisted action of $\G$ on $\A$ is a pair $(\alpha,\omega)$ of continuous maps $\alpha:\G\to{\rm Aut}({\A})$, $\omega:\G\times\G\to \mathcal{UM}({\A})$, such that \begin{itemize}
        \item[(i)] $\alpha_x(\omega(y,z))\omega(x,yz)=\omega(x,y)\omega(xy,z)$,
        \item[(ii)] $\alpha_x\big(\alpha_y(a)\big)\omega(x,y)=\omega(x,y)\alpha_{xy}(a)$,
        \item[(iii)] $\omega(x,\e)=\omega(\e,y)=1, \alpha_\e={\rm id}_{{\A}}$,
    \end{itemize} for all $x,y,z\in\G$ and $a\in\A$. 

The quadruple $(\G,\A,\alpha,\omega)$ is called a \emph{twisted $C^*$-dynamical system}. Given such a twisted action, one usually forms the so called \emph{twisted convolution algebra} $L^1_{\alpha,\omega}(\G,\A)$, consisting of all Bochner integrable functions $\Phi:\G\to\A$, endowed with the product 
\begin{equation}\label{convolution}
    \Phi*\Psi(x)=\int_\G \Phi(y)\alpha_y[\Psi(y^{-1}x)]\omega(y,y^{-1}x)\d y
\end{equation} and the involution \begin{equation}\label{involution}
    \Phi^*(x)=\omega(x,x^{-1})^*\alpha_x[\Phi(x^{-1})^*].
\end{equation} Making $L^1_{\alpha,\omega}(\G,\A)$ a Banach $^*$-algebra under the norm $\norm{\Phi}_{L^1_{\alpha,\omega}(\G,\A)}=\int_\G\norm{\Phi(x)}_{\A}\d x$. When the twist is trivial ($\omega\equiv1$), we omit any mention to it and call the resulting algebra $L^1_{\alpha}(\G,\A)$ as (simply) the \emph{convolution algebra} associated with the action $\alpha$. In this case, the triple $(\G,\A,\alpha)$ is called a (untwisted) \emph{$C^*$-dynamical system}.

The algebras mentioned above can be easily described as algebras of integrable cross-sections $L^1(\G\,\vert\,\mathscr C_\alpha)$, for particular Fell bundles. In fact the associated bundle may be described as $\mathscr C_\alpha=\A\times\G$, with quotient map $q(a,x)=x$, constant norms $\norm{\cdot}_{\mathfrak C_x}=\norm{\cdot}_{\A}$, and operations \begin{equation*}
    (a,x)\bu(b,y)=(a\alpha_x(b)\omega(x,y),xy)\quad\textup{and}\quad (a,x)^\bu=(\alpha_{x^{-1}}(a^*)\omega(x^{-1},x),x^{-1}).
\end{equation*} 
\end{ex}

We will now introduce the left regular representation of $L^1(\G\,\vert\,\mathscr C)$, as it allows us to get useful norm estimates and use $C^*$-theory. The space $L^2_\e(\G\,\vert\,\mathscr C)$ is the completion of $L^2(\G\,\vert\,\mathscr C)$ under the norm $$\norm{\Phi}_{L^2_\e(\G\,\vert\,\mathscr C)}=\norm{\int_\G\Phi(x)^\bu\bu\Phi(x)\,\d x}_{\mathfrak C_\e}^{1/2}.$$ This is a Hilbert $C^*$-module over $\mathfrak C_\e$, so the set of adjointable operators is a $C^*$-algebra under the operator norm. We denote this algebra by $\mathbb B_a(L^2_\e(\G\,\vert\,\mathscr C))$. The left regular representation $\lambda$ is then the $^*$-monomorphism given by \begin{equation*}
    \lambda:L^1(\G\,\vert\,\mathscr C)\to \mathbb B_a(L^2_\e(\G\,\vert\,\mathscr C)), \textup{ defined by }\lambda(\Phi)\Psi=\Phi*\Psi, \textup{ for all }\Psi\in L^2(\G\,\vert\,\mathscr C).
\end{equation*} For an amenable $\G$, ${\rm C^*}(\G\,\vert\,\mathscr C)$ coincides with $\overline{\lambda(L^1(\G\,\vert\,\mathscr C))}^{\norm{\cdot}_{\mathbb B_a(L^2_\e(\G\,\vert\,\mathscr C))}}$, cf. \cite{ExNg02}. The following lemma was proven in \cite[Lemma 3.4]{Fl24}.

\begin{lem} Let $\Psi\in L^2(\G\,\vert\,\mathscr C)$ and $\Phi\in L^2_e(\G\,\vert\,\mathscr C)$. Then \begin{equation}\label{young}
    \norm{\Psi*\Phi}_{L^\infty(\G\,\vert\,\mathscr C)}\leq \norm{\Psi}_{L^2(\G\,\vert\,\mathscr C)}\norm{\Phi}_{L^2_\e(\G\,\vert\,\mathscr C)}.
\end{equation}\end{lem}

As in \cite{Fl24}, the growth of $1$-parameter unitary groups plays a major role in the development of our results. In order to include the non-unital case, we are forced to consider the entire function $u:\mathbb C\to \mathbb C$, given by \begin{equation}\label{functions}
    u(z)=e^{iz}-1=\sum_{k=1}^{\infty} \frac{i^k z^k}{k!}. 
\end{equation} and replace $e^{i\Phi}$ by $u(i\Phi)$ to avoid unnecessary unitizations. The following lemma should be familiar to the reader, as it is similar to \cite[Lemma 3.5]{Fl24}, but different due to our somewhat different context. The proof, however, stays the same.

\begin{lem}\label{computation1}
    Let $\Phi\in L^1(\G\,\vert\,\mathscr C)\cap L^2(\G\,\vert\,\mathscr C)$, then \begin{equation}\label{computation}
        \norm{u(\Phi)}_{L^{2}_\e(\G\,\vert\,\mathscr C)}\leq \norm{\Phi}_{L^{2}_\e(\G\,\vert\,\mathscr C)}.
    \end{equation}
\end{lem}
\begin{proof}
    Define the entire function $w:\mathbb C\to \mathbb C$, given by $$w(z)=\frac{e^{iz}-1}{z}=\sum_{k=0}^{\infty} \frac{i^{k+1} z^{k}}{(k+1)!}.$$ It is clear that $u(z)=w(z)z$. We now note that $u(\Phi)=v\big(\lambda(\Phi)\big)\Phi$ and hence, if ${\rm Spec}(a)$ denotes the spectrum of an element $a$ in $\mathbb B_a(L^2_\e(\G\,\vert\,\mathscr C))$, we have \begin{align*}
        \norm{u(\Phi)}_{L^2_\e(\G\,\vert\,\mathscr C)}&\leq\norm{w\big(\lambda(\Phi)\big)}_{\mathbb B(L^2_\e(\G\,\vert\,\mathscr C))}\norm{\Phi}_{L^2_\e(\G\,\vert\,\mathscr C)} \\
        &=\sup_{\alpha\in {\rm Spec}(\lambda(\Phi))} |w(\alpha)|\,\norm{\Phi}_{L^2_\e(\G\,\vert\,\mathscr C)}\\
        &\leq\sup_{\alpha\in \mathbb R}|w(\alpha)|\,\norm{\Phi}_{L^2_\e(\G\,\vert\,\mathscr C)} \\
        &\leq \norm{\Phi}_{L^2_\e(\G\,\vert\,\mathscr C)},
    \end{align*} finishing the proof.
\end{proof}

Recalling the previous section, we note that a good deal of the assumptions require the existence of dense subalgebras with nice properties. For the $L^1$-algebra of a Fell bundle, we will construct these algebras using weights on the group $\G$. The relevant definitions are the following.
\begin{defn}
    A {\rm weight} on the locally compact group $\G$ is a measurable, locally bounded function $\nu: \G\to [1,\infty)$ satisfying 
\begin{equation*}\label{submultiplicative}
 \nu(xy)\leq \nu(x)\nu(y)\,,\quad\nu(x^{-1})=\nu(x)\,,\quad\forall\,x,y\in\G\,.
\end{equation*} In addition, the weight $\nu$ is said to be a \emph{polynomial weight} if there is a constant $C>0$ such that \begin{equation}\label{polyweight}
        \nu(xy)\leq C\big(\nu(x)+\nu(y)\big),
    \end{equation} for all $x,y\in \G$.
\end{defn}

\begin{rem}
    If $\G$ is of polynomial growth and compactly generated or discrete and locally finite, then it is possible to construct a polynomial weight $\nu$ on $\G$ such that $\nu^{-1}\in L^p(\G)$, for any $0<p<\infty$ \cite{Py82}.
\end{rem}

During the rest of the section, $\mathfrak E$ will denote the Banach $^*$-algebra $L^{1,\nu}(\G\,\vert\,\mathscr C)\cap L^\infty(\G\,\vert\,\mathscr C)$, endowed with the norm $$\norm{\Phi}_{\mathfrak E}=\max\{\norm{\Phi}_{L^{1,\nu}(\G\,\vert\,\mathscr C)},\norm{\Phi}_{L^\infty(\G\,\vert\,\mathscr C)}\}.$$ In fact, we recall the following lemma, stated and proved in \cite[Proposition 5.15, Lemma 4.12]{Fl24}.

\begin{lem}\label{gendiff1}
    Let $\nu$ a polynomial weight on $\G$ such that $\nu^{-1}$ belongs to $L^p(\G)$, for some $0<p<\infty$. Then $\mathfrak E$ is a symmetric Banach $^*$-subalgebra of $L^{1}(\G\,\vert\,\mathscr C)$. Moreover, there exist a constant $D\geq 1$ such that \begin{equation}\label{gendiff}
        \norm{\Phi^4}_{\mathfrak E}\leq D \norm{\Phi}_{L^{2}_\e(\G\,\vert\,\mathscr C)}^{1/(p+1)}\norm{\Phi}_{\mathfrak E}^{(4p+3)/(p+1)}.
    \end{equation}
\end{lem}

The next lemma was inspired on \cite[Lemma 3]{Py82}. It will be used soon to deduce the regularity of all the algebras $\mathfrak E(\Phi)$, for $\Phi\in \mathfrak E_{\rm sa}$.

\begin{lem}\label{asymp}
    Let $1<\gamma<4$ and let $\{a_n\}_{n=1}^\infty$ be a sequence of non-negative real numbers such that \begin{enumerate}
        \item[(i)] $a_{n+m}\leq a_na_m$ and
        \item[(ii)] $a_{4n}\leq na_n^\gamma,$
    \end{enumerate} for all $n$. Then for all $\tau>\log_4(\gamma)$, one has $a_n=O(e^{n^\tau}),\textup{ as }n\to\infty$.
\end{lem}

\begin{proof}
Because of \emph{(ii)}, we have $$a_{4^k}\leq 4^{\beta(k)}a_1^{\gamma^k},$$ where the sequence $\beta(n)$ satisfies $\beta(n+1)=\gamma \beta(n)+n$ and therefore $\beta(n)=\tfrac{\gamma^n-n\gamma+\gamma-1}{(\gamma-1)^2}$ and $$a_{4^k}\leq 4^{\tfrac{\gamma^k-k\gamma+\gamma-1}{(\gamma-1)^2}}a_1^{\gamma^k}\leq 4^{\tfrac{\gamma^k}{(\gamma-1)^2}}a_1^{\gamma^k}=4^{E\gamma^k},$$ with $E=\tfrac{1}{(\gamma-1)^2}+\log_4(a_1)$. For a general $n\in\N$, we consider its $4$-adic expansion $n=\sum_{k=0}^m \epsilon_k 4^k$, where $\epsilon_k\in\{0,1,2,3\}$, $\log_4(n)\leq  m<1+\log_4(n)$ and see that $$a_n\leq \prod_{k=0}^m a_{4^k}^{\epsilon_k}\leq \prod_{k=0}^m 4^{E\epsilon_k\gamma^k}\leq 4^{mE\gamma^m}\leq 4^{4E(1+\log_4(n))n^{\log_4(\gamma)} }, $$ from which the desired property follows.
\end{proof}

\begin{prop}\label{growth}
    Let $\nu$ be a polynomial weight on $\G$ such that such that $\nu^{-1}$ belongs to $L^p(\G)$, for $0<p<\infty$. Then for every $\Phi=\Phi^*\in \mathfrak E$, one has $$\norm{u(n\Phi)}_{\mathfrak E}=O(e^{n^\tau}),\textup{ as }n\to\infty,$$ for every $\tau>\log_4(\tfrac{4p+3}{p+1})$.
\end{prop}
\begin{proof}
    We will, of course, make use of Lemma \ref{asymp}. For that matter, let $\gamma=\tfrac{4p+3}{p+1}\in(1,4)$ and consider the sequence $a_n=C(\norm{u(n\Phi)}_{\mathfrak E}+1)$, with $C\geq 1$ to be determined later. One then has \begin{align*}
        a_{n+m}&=C(\norm{u(n\Phi)*u(m\Phi)+u(n\Phi)+u(m\Phi)}_{\mathfrak E}+1) \\
        &\leq C(\norm{u(n\Phi)}_{\mathfrak E}+1)(\norm{u(m\Phi)}_{\mathfrak E}+1) \leq a_na_m.
    \end{align*} To prove part \emph{(ii)}, we first consider the case $\norm{u(n\Phi)}_{\mathfrak E}\leq 1$. In this case,  \begin{align*}
        a_{4n}&\leq C(\norm{u(n\Phi)}_{\mathfrak E}^4 + 4 \norm{u(n\Phi)}_{\mathfrak E}^3 + 6\norm{u(n\Phi)}_{\mathfrak E}^2 + 4\norm{u(n\Phi)}_{\mathfrak E})  \\
        &\leq 15C \\
        &\leq 15C(\norm{u(n\Phi)}_{\mathfrak E}+1)^{\gamma}.
    \end{align*} So $a_{4n}\leq na_n^\gamma$, if $C^{\gamma-1}\geq 15$. Now, if $\norm{u(n\Phi)}_{\mathfrak E}> 1$, then \begin{align*}
        \norm{u(4n\Phi)}_{\mathfrak E}+1&=\norm{u(n\Phi)^4+4u(n\Phi)^3+6u(n\Phi)^2+4u(n\Phi)}_{\mathfrak E}+1 \\
        &\overset{\eqref{gendiff}}{\leq } (D \norm{u(n\Phi)}_{L^{2}_\e(\G\,\vert\,\mathscr C)}^{1/(p+1)} + 15)\norm{u(n\Phi)}_{\mathfrak E}^{(4p+3)/(p+1)}  \\
        &\overset{\eqref{computation}}{\leq } n(D\norm{\Phi}_{L^{2}_\e(\G\,\vert\,\mathscr C)}^{1/(p+1)}+15)\norm{u(n\Phi)}_{\mathfrak E}^\gamma.
    \end{align*} And setting $C^{\gamma-1}\geq D\norm{\Phi}_{L^{2}_\e(\G\,\vert\,\mathscr C)}^{1/(p+1)}+15 $ yields the result.
\end{proof}

\begin{rem}
    It is worth noting that the growth obtained here is significantly bigger than the one in \cite[Theorem 1.3]{Fl24} (it is no longer polynomial). The advantage here is that this property applies to all self-adjoint elements in $\mathfrak E$ and not only the ones with compact support. 
\end{rem}

\begin{prop}\label{regularsub}
    Let $\nu$ be a polynomial weight on $\G$ such that $\nu^{-1}$ belongs to $L^p(\G)$, for $0<p<\infty$. Then for every $\Phi= \Phi^*\in\mathfrak E$, $\mathfrak E(\Phi)$ is a regular Banach function algebra. 
\end{prop}
\begin{proof}
    Because of Proposition \ref{growth}, we can choose some $\tau\in(0,1)$ such that $\norm{u(n\Phi)}_{\mathfrak E}=O(e^{n^\tau})$ as $n\to\infty$. This easily implies $\norm{u(t\Phi)}_{\widetilde{\mathfrak E}}=O(e^{{|t|}^\tau})$ as $|t|\to\infty$. This implies that, for some $C>0$, $$\int_{\mathbb R}\frac{\log(\norm{u(t\Phi)}_{\widetilde{\mathfrak E}})}{1+t^2}\d t\leq C\int_{\mathbb R}\frac{|t|^\tau}{1+t^2}\d t<\infty.$$ Hence, by a classical criterion of Shilov (see \cite[Example 2.4]{Ne92} for a short proof, written in english), $\widetilde{\mathfrak E}(\Phi)=\mathfrak E(\Phi)$ is regular.
\end{proof}

Proposition \ref{regularsub} implies the $^*$-regularity of $L^1(\G\,\vert\,\mathscr C)$ and so it allows us to state the following corollary. It will show that $L^1(\G\,\vert\,\mathscr C)$ is another algebra satisfying the conjecture of Albrecht and Dales \cite[pag. 380]{AlDa83}, which can also be found in \cite{Run94}.

\begin{cor}\label{noncont}
    Let $\nu$ be a polynomial weight on $\G$ such that $\nu^{-1}$ belongs to $L^p(\G)$, for $0<p<\infty$. Assume the continuum hypothesis. If there exists $n\in\N$ such that $L^1(\G\,\vert\,\mathscr C)$ has infinitely many inequivalent, topologically irreducible, $n$-dimensional $^*$-representations, then $L^1(\G\,\vert\,\mathscr C)$ is the domain of a discontinuous homomorphism into a Banach algebra.
\end{cor}
\begin{proof}
    Proposition \ref{regularsub} implies that $L^1(\G\,\vert\,\mathscr C)$ is locally regular. By \cite{Ba81}, $L^1(\G\,\vert\,\mathscr C)$ is $^*$-regular and then \cite[Theorem 1]{Ru96c} gives the result.
\end{proof}

We now turn our attention to the study of sets of synthesis in $L^1(\G\,\vert\,\mathscr C)$. We start with the following lemma, inspired by \cite[Lemma 10]{Ba83}.

\begin{lem}\label{barnes}
    Let $\G$ be a group of polynomial growth. Let $I$ a closed two-sided ideal of $L^1(\G\,\vert\,\mathscr C)$ with finite codimension, containing a bounded left approximate identity. If $\Phi\in I$ and $\epsilon>0$, then there exist $\Psi_1,\Psi_2\in I$ such that $$\norm{\Phi-\Psi_1*\Phi}_{L^1(\G\,\vert\,\mathscr C)}<\epsilon\quad\text{ and }\quad\Psi_2*\Psi_1=\Psi_1. $$
\end{lem}

\begin{proof}
    Because of density, we can find a finite dimensional subspace $X\subset C_{\rm c}(\G\,\vert\,\mathscr C)$, such that $I+X=L^1(\G\,\vert\,\mathscr C).$ If $P$ is the projection of $L^1(\G\,\vert\,\mathscr C)$ onto $I$ determined by this decomposition, then $P$ maps $C_{\rm c}(\G\,\vert\,\mathscr C)$ onto itself and hence $I \cap C_{\rm c}(\G\,\vert\,\mathscr C)$ is dense in $C_{\rm c}(\G\,\vert\,\mathscr C)$. This means that $I $ has a bounded left approximate identity $\{\Psi_\alpha\}$,  contained in  $I\cap C_{\rm c}(\G\,\vert\,\mathscr C)$. Without loss of generality, $\Psi_\alpha=\Psi_\alpha^*$ and $\norm{\Psi_\alpha}\leq C$.

    Now, because of \cite[Theorem 3.8]{Fl24}, there is a smooth functional calculus in $L^1(\G\,\vert\,\mathscr C)$ and $f(\Phi)\in I$ for all $f\in C_{\rm c}^{\infty}(\R)$ and $\Phi\in I$. When $f(t)=t^n$ in a neighborhood of $[-C,C]$, we have $f(\Phi)=\Phi^n$, and hence, because of \cite[Lemme 8]{Di60} and \cite[Theorem 2.5]{Fl24}, there exists a natural number $n\in\N$ and a sequence of functions $f_k\in C_{\rm c}^{\infty}(\R)$ with supports contained in a common compact set $K\subset\R$, so that \begin{equation*}
        \lim_{k\to\infty}\norm{f_k(\Psi)-\Psi^{n}}_{L^1(\G\,\vert\,\mathscr C)}=0,
    \end{equation*} for all $\Psi\in I$. On the other hand, if $g\in C_{\rm c}^{\infty}(\R)$ is such that $g\equiv 1$ on a neighborhood of $K$, then \begin{equation*}
        g(\Psi_\alpha)*f_k(\Psi_\alpha)=(g\cdot f_k)(\Psi_\alpha)=f_k(\Psi_\alpha).
    \end{equation*}  So \begin{align*}
        \norm{f_k(\Psi_\alpha)*\Phi-\Phi}_{L^1(\G\,\vert\,\mathscr C)}&\leq \norm{\big(f_k(\Psi_\alpha)-\Psi_\alpha^n\big)*\Phi}_{L^1(\G\,\vert\,\mathscr C)}+\norm{\Psi_\alpha^n*\Phi-\Phi}_{L^1(\G\,\vert\,\mathscr C)} \\
        &\leq \norm{f_k(\Psi_\alpha)-\Psi_\alpha^n}_{L^1(\G\,\vert\,\mathscr C)}\norm{\Phi}_{L^1(\G\,\vert\,\mathscr C)}+n\norm{\Psi_\alpha*\Phi-\Phi}_{L^1(\G\,\vert\,\mathscr C)},
    \end{align*} since $$\norm{\Psi_\alpha^n*\Phi-\Phi}_{L^1(\G\,\vert\,\mathscr C)}\leq \sum_{j=1}^n \norm{\Psi_\alpha^j*\Phi-\Psi_\alpha^{j-1}*\Phi}_{L^1(\G\,\vert\,\mathscr C)}\leq n\norm{\Psi_\alpha*\Phi-\Phi}_{L^1(\G\,\vert\,\mathscr C)}.$$ Therefore choosing $\Phi_1= f_k(\Psi_\alpha)$ and $\Phi_2= g(\Psi_\alpha)$ for large enough $k$ and $\alpha$ does the trick.
\end{proof}

\begin{thm}\label{synth}
    Suppose $\G$ is of polynomial growth, that $L^1(\G\,\vert\,\mathscr C)$ is symmetric and that every closed two-sided ideal $I$ of $L^1(\G\,\vert\,\mathscr C)$ has a bounded left approximate identity. Let $F$ be a finite closed subset of $\,{\rm Prim}_*L^1(\G\,\vert\,\mathscr C)$ such that every $P\in F$ has finite codimension. Then $F$ is a set of synthesis.
\end{thm}

\begin{proof}
    Define $$M=\{\Psi\in k(F)\mid \exists \Phi \in k(F) \text{ such that }\Phi*\Psi=\Psi\}.$$ Lemma \ref{barnes} applied to the finite codimensional ideal $k(F)$ implies that $h(M)=F$. It also clear that $F\subset h(\Phi)$ for all the elements $\Phi$ involved in the definition of $M$. Now, if $I$ is a closed two-sided ideal of $L^1(\G\,\vert\,\mathscr C)$ with $h(I)=F$, then because of \cite[Lemma 2]{Lu80}, we have $M\subset I$. Then applying Lemma \ref{barnes} again, we have $k(F)\subset I$ and hence $I=k(F)$.
\end{proof}

The work we have done so far allows us to check in fairly great generality the conditions given in Theorem \ref{theprop}, so we can state the following result -our main result-. It provides a fairly checkable criterion for automatic continuity. 

\begin{thm}\label{main1}
    Let $\G$ be a locally compact group that admits a polynomial weight $\nu$, such that $\nu^{-1}\in L^p(\G)$. Suppose that the algebra $\B=L^1(\G\,\vert\,\mathscr C)$ is symmetric and that every closed two-sided ideal $I$ of $\B$ has a bounded left approximate identity. Let $\X$ be a weak Banach $\B$-bimodule and $\theta:\B\to\X$ a $\B$-intertwining operator. Then $\theta$ is continuous if and only if $\mathscr I(\theta)$ is closed. In such a case, $\mathscr I(\theta)$ is of finite codimension. 
\end{thm}
\begin{proof}
    It follows directly from checking the conditions in Theorem \ref{theprop}. \emph{(i)} holds by Proposition \ref{regularsub}, \emph{(ii)} holds by \cite[Theorem 5.14]{Fl24},  \emph{(iii)} is guaranteed by Theorem \ref{synth} and \emph{(iv)} is assumed.
\end{proof}

\begin{cor}\label{main2}
    Let $\B=L^1(\G\,\vert\,\mathscr C)$ satisfying the same conditions as in Theorem \ref{main1}, if $\X$ is a Banach $\B$-bimodule, then $\theta:\B\to\X$ is automatically continuous. In particular, all derivations of $\B$ into Banach $\B$-bimoules are continuous.
\end{cor}
\begin{proof}
    Follows from Theorem \ref{main1} and the fact that $\mathscr I(\theta)$ is closed because $\X$ be a Banach $\B$-bimodule.
\end{proof}

On the other hand, Corollary \ref{theprop-simple} provides us with the following automatic continuity result. 

\begin{thm}\label{main3}
    Let $\G$ be a discrete group that admits a polynomial weight $\nu$, such that $\nu^{-1}\in \ell^p(\G)$. Suppose that the algebra $\B=\ell^1(\G\,\vert\,\mathscr C)$ is unital, symmetric and $C^*(\G\,\vert\,\mathscr C)$ has no proper closed two-sided ideals with finite codimension. Let $\X_1$ a Banach $\B$-bimodule, $\X_2$ a weak Banach $\B$-bimodule and $\theta:\X_1\to\X_2$ a $\B$-intertwining operator. Then $\theta$ is continuous. In particular, all algebra homomorphisms with domain $\B$ are automatically continuous.
\end{thm}

Let us now provide examples of algebras satisfying the conditions above. The following two remarks provide known results from the pre-existing literature.

\begin{rem}\label{ex-sym}
    $L^1(\G\,\vert\,\mathscr C)$ is known to be symmetric, irrespective of the Fell bundle $\mathscr C$, as soon as $\G$ is nilpotent \cite{FJM} or compact \cite[Theorem 7.5]{Fl24}.
\end{rem}

\begin{rem}\label{ex-am}
As per the results in \cite{dJeHP17}, the algebras $\ell^1_\alpha(\G,\A)$, corresponding to $C^*$-dynamical systems where $\G$ is amenable and $\A$ is finite-dimensional or abelian, are amenable. Because of Remark \ref{conditions}, this is enough to guarantee that every closed two-sided ideal $I$ has a bounded left approximate identity.
\end{rem}

\begin{rem}
    Let $(\G,\A,\alpha)$ be a $C^*$-dynamical system, with $\G$ discrete. Then $C^*(\ell^1_\alpha(\G,\A))=\A\rtimes_\alpha\G$ is simple whenever the action $\alpha$ is minimal and topologically free \cite{AS94}.
\end{rem}

\begin{rem}
    If $\alpha$ is now a partial action of a discrete group $\G$ on the locally compact space $X$, then minimality and topological freeness of $\alpha$ also guarantee the simplicity of $\A\rtimes_\alpha\G$ \cite[Corollary 29.8]{Ex17}. More interesting is the fact that $\A\rtimes_\alpha\G$ has no proper finite-codimensional closed two-sided ideals, provided that $\G$ is amenable, infinite and that the restriction of $\alpha$ to any invariant subset is topologically free \cite[Theorem 29.9]{Ex17}.
\end{rem}

\begin{cor}\label{segundo}
    Let $\G$ be either compact or nilpotent and compactly generated. Let $\X$ be a Banach $\B$-bimodule and $\theta:\B\to\X$ a $\B$-intertwining operator. Then $\theta$ is automatically continuous for the following choices of $\B$: \begin{enumerate}
        \item[(i)] Twisted group algebras $L^1_\omega(\G)$, associated with a $2$-cocycle $\omega:\G\times \G\to \CC$.
        \item[(ii)] Convolution algebras $\ell^1_\alpha(\G,\A)$, where $(\G,\A,\alpha)$ is a $C^*$-dynamical system with $\A$ either an abelian $C^*$-algebra or a finite-dimensional $C^*$-algebra.
        \item[(iii)] $\ell^1(\G\,\vert\,\mathscr C)$, if it is unital and $C^*(\G\,\vert\,\mathscr C)$ has no proper finite-codimensional closed two-sided ideals.
        \item[(iv)] $\ell^1_\alpha(\G,C(X))$, if $\G$ is infinite and $\alpha$ is a free partial action on the compact space $X$, such that the restriction of $\alpha$ to any invariant subset of $X$ is topologically free.
    \end{enumerate} In particular, all derivations into Banach bimodules and with domains in the previously mentioned algebras are continuous. \end{cor}
\begin{proof}
    It all follows from Corollary \ref{main2} and the previous discussion, except for \emph{(i)}, which requires Theorem \ref{twistedgp}.
\end{proof}

\begin{cor}\label{tercero}
    Let $\G$ be infinite and either locally finite or of polynomial growth and finitely generated. Let $\B=\ell^1_\alpha(\G,C(X))$, where $\alpha$ is a (partial) action on the compact space $X$. Further suppose that the restriction of $\alpha$ to any invariant subset of $X$ is topologically free. Let $\X_1$ a Banach $\B$-bimodule, $\X_2$ a weak Banach $\B$-bimodule and $\theta:\X_1\to\X_2$ a $\B$-intertwining operator. Then $\theta$ is automatically continuous. In particular, all homomorphisms $\varphi:\B\to\A$, where $\A$ is a Banach algebra, are continuous. 
\end{cor}

We finalize this section with conditions that relate the automatic continuity properties of $L^1(\G\,\vert\,\mathscr C)$ to its $C^*$-completion. It is proven directly by applying Theorem \ref{cstar}.

\begin{thm}\label{cstar-cor}
    Let $\G$ be a locally compact group that admits a polynomial weight $\nu$, such that $\nu^{-1}\in L^p(\G)$. Suppose that the algebra $\B=L^1(\G\,\vert\,\mathscr C)$ is symmetric and that every closed two-sided ideal $I$ of $\B$ has a bounded left approximate identity. Then the following are equivalent for every homomorphism $\theta:C^*(\G\,\vert\,\mathscr C)\to \A$, where $\A$ is another Banach algebra.\begin{enumerate}
        \item[(a)] $\theta$ is continuous.
        \item[(b)] $\theta|_\B$ is continuous.
        \item[(c)] $\mathscr I(\theta)\cap \B$ is closed in $\B$. 
    \end{enumerate}
\end{thm}

\begin{rem}
    Theorem \ref{cstar-cor} applies, of course, to all the algebras mentioned in Corollary \ref{segundo}.
\end{rem}

\section{Appendix: On finite-codimensional ideals in twisted group algebras}\label{COMPL}

On this appendix we will show that finite codimensional, closed two-sided ideals in a twisted group algebra $L_\omega^1(\G)$ have bounded left-approximate identities as soon has $\G$ is amenable. This helps us bypass the assumptions of amenability previously made. To my best knowledge, it is not known if these algebras are amenable. The technique used here was introduced in \cite{LvJ73}.

From now on, fix a locally compact group $\G$ (not necessarily unimodular) and a complex-valued $2$-cocycle, in the sense of Example \ref{mainex}. Let $\Phi\in L^1_\omega(\G)$ and $y\in \G$. The formulas \begin{align}
    U_y \Phi(x)=\Phi(y^{-1}x)\omega(y,y^{-1}x),&\quad  L_y\Phi(x)=\Phi(y^{-1}x) \\
    V_y \Phi(x)=\Phi(xy^{-1})\omega(xy^{-1},y)\Delta(y^{-1}),&\quad  R_y\Phi(x)=\Phi(xy)\Delta(y)
\end{align} define isometries in $\mathbb B(L^1_\omega(\G))$ such that \begin{equation}
    \Phi*\Psi=\int_\G \Phi(y) U_y\Psi\,\d y=\int_\G \Psi(y) V_y\Phi \,\d y
\end{equation} Furthermore, we have the relations \begin{align}\label{formulas}
    U_xU_y=\omega(x,y)U_{xy},&\qquad U_y(\Phi*\Psi)=U_y(\Phi)*\Psi \\
    V_yV_x=\omega(x,y)V_{xy}\quad&\text{and}\quad V_y(\Phi*\Psi)=\Phi*V_y(\Psi)
\end{align} for all $x,y\in \G$ and $\Phi,\Psi\in L^1_\omega(\G)$. In particular, the first identity implies that \begin{equation}
    \omega(x,x^{-1})=\omega(x^{-1},x)\quad\text{ and }\quad U_y^{-1}=\overline{\omega(x,x^{-1})}U_{x^{-1}}.
\end{equation} An immediate consequence is the following lemma.

\begin{lem}\label{multiplier}
    Let $I\subset L^1_\omega(\G)$ be a closed left (resp. right) ideal. Then $U_y I\subset I$ (resp. $V_y I\subset I$), for all $y\in \G$.
\end{lem}

\begin{proof}
    Let $\Psi_\gamma \in L^1_\omega(\G)$ be a bounded approximate identity. Then for all $\Phi\in I$, we have $$I \ni U_y(\Psi_\gamma)*\Phi=U_y(\Psi_\gamma *\Phi)\to U_y\Phi$$ and hence $U_y\Phi\in I$. The case of right ideals is analogous. 
\end{proof}

In any case, we want to take advantage of the formulas in \eqref{formulas} to perturbe other operators in $\mathbb B(L^1_\omega(\G))$ and that is accomplished by the following lemma.

\begin{lem}\label{comm}
    For every $T\in \mathbb B(L^1_\omega(\G))$ and $\Psi\in L^1_\omega(\G)$, the formula \begin{equation}
        T_\Phi\Psi=\int_\G \Phi(y) U_y T U_{y}^{-1}\Psi\,\d y
    \end{equation} defines a bounded operator $T_\Phi\in \mathbb B(L^1_\omega(\G))$, of norm at most $\norm{\Phi}_{L^1_\omega(\G)}\norm{T}_{\mathbb B(L^1_\omega(\G))}$ and that satisfies \begin{equation}\label{idk}
        \norm{\Phi_1*T_\Psi\Phi_2-T_\Psi(\Phi_1*\Phi_2)}_{L^1_\omega(\G)}\leq \norm{\Phi_2}_{L^1_\omega(\G)}\norm{T}_{\mathbb B(L^1_\omega(\G))}\int_\G |\Phi_1(y)| \norm{\Psi-L_{y^{-1}}\Psi}_{L^1_\omega(\G)}\,\d y.
    \end{equation} Moreover, if $\G$ is amenable, then for all $\Phi_1,\ldots, \Phi_N\in L^1_\omega(\G)$ and $\epsilon>0$, there exists a positive funciton $\Psi \in L^1_\omega(\G)$ of integral $1$, such that for all $n=1,\ldots, N$, \begin{equation}\label{idk2}
        \norm{\Phi_n*T_\Psi\Phi_0-T_\Psi(\Phi_n*\Phi_0)}_{L^1_\omega(\G)}\leq \epsilon\norm{\Phi_0}_{L^1_\omega(\G)}\norm{T}_{\mathbb B(L^1_\omega(\G))}
    \end{equation} for all $T\in \mathbb B(L^1_\omega(\G))$ and $\Phi_0 \in L^1_\omega(\G)$. \end{lem}

\begin{proof}
    Since $U_y$ and $U_y^{-1}$ are isometries, $\norm{U_y}_{ \mathbb B(L^1_\omega(\G))}=\norm{U_y^{-1}}_{ \mathbb B(L^1_\omega(\G))}=1$ and $$\norm{T_\Phi}_{ \mathbb B(L^1_\omega(\G))}\leq \norm{\Phi}_{L^1_\omega(\G)}\norm{T}_{\mathbb B(L^1_\omega(\G))}$$ follows easily. In order to continue with the proof, let us note that \begin{align*}
        T_\Psi U_y \Phi_2&=\int_\G \Psi(x) U_x TU_{x}^{-1}U_y\Phi_2\,\d x \\
        &=\int_\G \overline{\omega(x,x^{-1})}\omega(x^{-1},y)\Psi(x) U_x TU_{x^{-1}y}\Phi_2\,\d x \\
        &= \int_\G \overline{\omega(yz,z^{-1}y^{-1})}\Psi(yz)U_{yz}TU_{z^{-1}}\Phi_2\,\d z \\
        &= U_y\int_\G \overline{\omega(yz,z^{-1})}\,\overline{\omega(y,z)}\omega(z,z^{-1})\Psi(yz)U_{z}TU_{z}^{-1}\Phi_2\,\d z \\
        &=U_y \int_\G \Psi(yz) U_{z}TU_{z}^{-1}\Phi_2\,\d z =U_y T_{L_{y^{-1}}\Psi}\Phi_2
    \end{align*} and hence \begin{align*}
        \Phi_1*T_\Psi\Phi_2-T_\Psi(\Phi_1*\Phi_2)&=\int_G \Phi_1(y) U_yT_\Psi\Phi_2\,\d y -  T_\Psi\big(\int_G \Phi_1(y) U_y\Phi_2\,\d y\big)   \\
        &=\int_\G \Phi_1(y)\big(U_yT_\Psi- T_\Psi U_y\big) \Phi_2\,\d y \\
        &=\int_\G \Phi_1(y) U_y \big(T_\Psi-T_{L_{y^{-1}}\Psi})\Phi_2\,\d y
    \end{align*} so the inequality \eqref{idk} follows. The rest of the proof follows as in \cite[Lemma 1]{LvJ73}.
\end{proof}

\begin{thm}\label{twistedgp}
    Let $\G$ be an amenable locally compact group. Let $I$ be a closed left (resp. right) ideal of $L^1_\omega(\G)$ such that there is a continuous projection of $L^1_\omega(\G)$ onto $I$. Then $I$ contains a bounded right (resp. left) approximate identity of norm at most $\norm{P}_{\mathbb B(L^1_\omega(\G))}$.
\end{thm}
\begin{proof}
    We focus on the case of left ideals, the case of right ideals is analogous. Given $y\in\G$, we note that $U_y T U_{y}^{-1}$ maps $L^1_\omega(\G)$ onto $I$ (Lemma \ref{multiplier}) and hence $U_y T U_{y}^{-1}\Phi=\Phi$, for all $\Phi\in I$. Hence $P_\Psi$ is a projection of $L^1_\omega(\G)$ onto $I$ if $\int_G \Psi(x)\,\d x=1$. 
    
    Because of \eqref{idk2}, there exists a projection $Q=P_\Psi$ with norm at most $\norm{P}_{\mathbb B(L^1_\omega(\G))}$ and such that for all $\Phi_1,\ldots, \Phi_N\in L^1_\omega(\G)$ and $\epsilon>0$, $$\norm{\Phi_n*Q\,\Phi_0-Q(\Phi_n*\Phi_0)}_{L^1_\omega(\G)}\leq \epsilon\norm{\Phi_0}_{L^1_\omega(\G)}.$$ We now choose $\Phi_0$ of norm $1$ and such that $\norm{\Phi_n*\Phi_0-\Phi_n}_{L^1_\omega(\G)}\leq \epsilon/\norm{Q}_{\mathbb B(L^1_\omega(\G))}$ for each $n=1,\ldots, N$. Then \begin{align*}
        \norm{\Phi_n*Q\,\Phi_0-\Phi_n}_{L^1_\omega(\G)}&\leq \norm{\Phi_n*Q\,\Phi_0-Q(\Phi_n*\Phi_0)}_{L^1_\omega(\G)}+\norm{Q(\Phi_n*\Phi_0)-\Phi_n}_{L^1_\omega(\G)} \\
        &\leq  2\epsilon.
    \end{align*} Finishing the proof.
\end{proof}

\section*{Acknowledgements}

The author has been partially supported by the NSF grant DMS-2000105. He is also grateful to professors Ben Hayes and Hannes Thiel for their helpful comments.

\printbibliography

\bigskip
\bigskip
ADDRESS

\smallskip
Felipe I. Flores

Department of Mathematics, University of Virginia,

114 Kerchof Hall. 141 Cabell Dr,

Charlottesville, Virginia, United States

E-mail: hmy3tf@virginia.edu

\end{document}